\documentclass{gSSR2e-anon}

\newcommand{\Lp}{L}
\newcommand{\Hg}{\mathrm{H}}
\newcommand{\RHg}{\mathrm{RH}}
\newcommand{\NLSE}{\mathop{\mathrm{NLSE}}}
\newcommand{\LSI}{\mathop{\mathrm{LSI}}}
\newcommand{\PW}{\mathop{\mathrm{PW}}}
\newcommand{\RPW}{\mathop{\mathrm{RPW}}}
\newcommand{\svec}[2]{\left[ \begin{array}{c} #1 \\ #2 \end{array} \right]}
\newcommand{\mC}{\mathbb{C}}
\newcommand{\mR}{\mathbb{R}}
\newcommand{\mT}{\mathbb{T}}
\newcommand{\mZ}{\mathbb{Z}}

\begin{document}

\title{Logarithmic Sobolev inequalities and spectral concentration for the cubic Schr\"odinger equation}


\author{Gordon Blower$^{\rm a}$$^{\ast}$\thanks{$^\ast$Corresponding author. Email: g.blower@lancaster.ac.uk
\vspace{6pt}},
Caroline Brett$^{\rm a}$
and Ian Doust$^{\rm b}$\\\vspace{6pt}  $^{\rm a}${\em{Department of Mathematics and Statistics, Lancaster University, Lancaster, England}}; $^{\rm b}${\em{School of Mathematics and Statistics, University of New South Wales, Sydney, Australia}}
\vspace{6pt}
\received{version 1.1 12 August 2013 }
}

\maketitle

\begin{abstract} The nonlinear Schr\"odinger equation $\NLSE(p, \beta)$,
$-iu_t=-u_{xx}+\beta \vert u\vert^{p-2} u=0$,
arises from a Hamiltonian on infinite-dimensional phase space
$\Lp^2(\mT)$. For $p\leq 6$, Bourgain (Comm. Math. Phys. 166 (1994), 1--26) has shown that there exists a Gibbs measure $\mu^{\beta}_N$ on balls $\Omega_N=
\{ \phi\in \Lp^2(\mT) \,:\, \Vert \phi \Vert^2_{\Lp^2} \leq N\}$
in phase space such that the Cauchy problem for $\NLSE(p,\beta)$ is well posed on the support of $\mu^{\beta}_N$, and that $\mu^{\beta}_N$ is invariant under
the flow.
This paper shows that $\mu^{\beta}_N$ satisfies a logarithmic Sobolev inequality for the focussing case $\beta <0$ and $2\leq p\leq 4$
on $\Omega_N$ for all $N>0$; also $\mu^{\beta}$ satisfies a restricted LSI for $4\leq p\leq 6$ on compact subsets of $\Omega_N$ determined by H\"older
norms. Hence for $p=4$, the spectral data of the periodic Dirac operator
in $\Lp^2(\mT; \mC^2)$ with random potential $\phi$ subject to $\mu^{\beta}_N$ are concentrated near to their mean values. The paper concludes with 
a similar result for the spectral data of Hill's equation when the potential is random and subject to the Gibbs measure of KdV.

\begin{classcode} 37L50, 35QK53 \end{classcode}
\end{abstract}

\section{Introduction}
This paper is concerned with the statistical mechanics of families of solutions for partial differential equations.
The periodic nonlinear Schr\"odinger equation studied is 
    \begin{equation}\label{1.3}
   \NLSE(p,\beta): \qquad
   -i  \frac{\partial u}{\partial t}
      =- \frac{\partial^2 u}{\partial x^2} +\beta \vert u(x,t) \vert^{p-2} u(x,t),
  \end{equation}
where $u: \mT \times \mR \to \mC$, $p \ge 2$ and $\beta \in \mR$.  The associated Cauchy problem specifies the initial condition $u(x,0)=\phi (x)$, where $\phi$ is periodic.

Invariant measures associated with this equation have been studied by several authors, including Lebowitz, Rose and Speer \cite{LRS}, and Bourgain \cite{B,B1}. Bourgain \cite{B}
showed that for $p \le 6$, that 
there exists a Gibbs measure $\mu^{\beta}_N$ on balls $\Omega_N=
\{ \phi\in \Lp^2(\mT) \,:\, \Vert \phi \Vert^2_{\Lp^2} \leq N\}$
in phase space such that the Cauchy problem for $\NLSE(p,\beta)$ is well posed on the support of $\mu^{\beta}_N$, and that $\mu^{\beta}_N$ is invariant under
the flow. This paper shows that $\mu^{\beta}_N$ satisfies a logarithmic Sobolev inequality for the focussing case $\beta <0$ and $2\leq p\leq 4$
on $\Omega_N$ for all $N>0$.  In the case of $4\leq p\leq 6$, we show that $\mu^{\beta}_N$ satisfies a restricted LSI on compact subsets of $\Omega_N$ determined by 
certain H\"older norms.

A consequence of this result is that the spectral data of the periodic Dirac operator
in $\Lp^2(\mT; \mC^2)$ with random potential $\phi$ subject to $\mu^{\beta}_N$ for $p=4$ are concentrated near to their mean values. 
The equation $\NLSE(p,\beta)$ arises from the Hamiltonian
    \begin{equation}\label{1.1}
  H(\phi)=\frac{1}{2} \int_{\mT}
    \bigl\vert \phi'(x) \bigr\vert^2 \,\frac{dx}{2\pi}
        +\frac{\beta}{p}
        \int_{\mT} \vert \phi(x) \vert^p \frac{dx}{2\pi}.
  \end{equation}
There is a sequence of invariants under the motion, including the number operator
  \begin{equation}\label{1.4}
   N = \int_{\mT} \vert u(x,t) \vert^2 \frac{dx}{2\pi}
  \end{equation}
and the Hamiltonian $H(u(\cdot,t))$ itself. In much of what follows, it will be more convenient to work with the coordinates of elements of $\Lp^2(\mT; \mC)$ with respect to the standard Fourier basis. 
That is, we shall consider the $\Lp^2$ function $\phi(x) =\sum_{n=-\infty}^\infty (a_n+ib_n)e^{inx}$ with real Fourier coefficients $((a_j,b_j))_{j=-\infty}^\infty$ in $\ell^2$. 
One can then readily verify that, with 
  \begin{equation}\label{eqn4}
   H(\phi)=\frac{1}{2} \sum_{n=-\infty}^\infty n^2(a_n^2+b_n^2)
        + \frac{\beta}{p} \int_{\mT}
             \,\bigl\vert \sum_{n=-\infty}^\infty (a_n+ib_n)e^{inx}\bigr\vert^p
                 \,\frac{dx}{2\pi},
  \end{equation}
\noindent the $((a_j,b_j))_{j=0}^\infty$ give a system of canonical coordinates. In order to obtain a normalized Gibbs measure, one restricts 
attention to bounded subsets of $\Lp^2(\mT; \mC)$, and forms the modified canonical ensemble with phase space on the ball
     \begin{equation} \Omega_N = \Bigl\{ \phi\in \Lp^2(\mT; \mC) \,:\,
                            \int_{\mT} \vert \phi (x)\vert^2 {\frac{dx}{2\pi}} \leq N\Bigr\}
     \end{equation}
\noindent for fixed $N>0$. We then seek to define the Gibbs measure on $\Omega_N$ via
  \begin{equation}\label{1.5}
 \mu^{\beta}_N= Z(p, \beta ,N)^{-1} \exp \bigl( -H(\phi) \bigr)
          \ {\bf I}_{\Omega_N}(\phi )\!\! \prod_{x\in [0,2\pi ]} d^2\phi (x).
  \end{equation}
To make this more precise, let $(\zeta_j, \zeta_j')_{j=-\infty}^\infty$  be mutually independent $N(0,1)$ 
Gaussian random variables, and let Wiener loop $\mathcal{W}$ be the measure that is
induced on $\Lp^2({\mT}; {\mC})$ by the random Fourier series
  \[ \omega \mapsto \phi_\omega(x) = \sum_{j=-\infty ;j\neq 0}^\infty \frac{\zeta_j+i\zeta'_j}{j} e^{ijx} . \]
Thus $\mathcal{W}$ provides an interpretation of 
$\exp\bigl( -\frac{1}{2} \int \vert \phi'(x) \vert^2 \, dx \bigr) \prod_{x\in {\mT}}d^2\phi (x)$. In  terms of the 
 Fourier coefficients this may be written as the measure on $\ell^2$ given by the product
$\prod_{j=-\infty}^\infty \exp\bigl(-\frac{1}{2} j^2(a_j^2+b_j^2) \bigr) j^2\, \frac{da_jdb_j}{2\pi }$.

Dealing with the other term of the Hamiltonian from \ref{eqn4} is more delicate. By results of \cite{B} and \cite{LRS}, the factor $\exp\bigr( -\frac{\beta}{p} \int \vert \phi (x)\vert^p dx \bigr)$ is
integrable over $\Omega_N$ with respect to Wiener loop for
all $\beta\in {\mR}$, $0 < N < \infty$ and $2 \leq p < 6$. Here $\beta <0$ gives the focussing case, where $p=6$ is marginal for there to exist a normalized Gibbs measure; see \cite{LRS}.

By a cylindrical function, we mean $F:\ell^2\rightarrow {\mR}$ of the form $F(x)=f((x_j)_{j=-m}^m)$ for some $m<\infty$, 
where $f:{\mR}^{2m+1} \rightarrow {\mR}$ is a $C^\infty$ function of compact support in the sense of calculus. 

\begin{definition} 
A probability measure $\nu$ on a Borel subset $\Omega$ of  $\ell^2$ is said to satisfy a \emph{logarithmic Sobolev inequality} (LSI) with constant $\alpha >0$ if
  \begin{equation}\label{1.6}
\LSI (\alpha ):\qquad  \int_\Omega F(x)^2\log \Bigl( F(x)^2/\int F^2\, d\nu \Bigr)\, \nu (dx)
        \leq \frac{2}{\alpha} \int_\Omega \Vert\nabla  F(x)\Vert^2_{\ell^2}\, \nu (dx)
  \end{equation}
for all cylindrical functions $F$.
\end{definition}

Such a logarithmic Sobolev inequality is used to derive concentration inequalities for Lipschitz functions. If $F:\Omega\rightarrow {\mR}$ is such that $\vert F(x)-F(y)\vert\leq 
\Vert x-y\Vert_{\ell^2}$ 
for all $x,y\in \ell^2$, and also $\int_\Omega F(x)\nu (dx)=0$, then $\int_{\Omega}e^{tF(x)}\nu (dx)\leq e^{t^2/(2\alpha )}$ for all $t\in {\mR}$; so $F$ is concentrated close to 
its mean value. 
In this respect, $\nu$ resembles a Gaussian measure, and the significant examples live on $K_\sigma$ subsets of the infinite-dimensional Hilbert space $\ell^2$. 
In section 2 of this paper,
the measures live on the ball $\Omega_N$ in $\ell^2$, 
and then in sections 3 and 4 we consider random variables which are Lipschitz functions on $\Omega_N$.\par

\indent In order to state the main technical result of this paper, we need some additional notation.
Let $\ell^2 =\ell^2({\mZ};{\mR})$ have the usual inner product. For $\gamma \geq 0$ we introduce the H\"older spaces
  \[
  {\Hg}^\gamma =\Bigl\{ \phi (x)=\sum_{n=-\infty }^\infty \phi_n e^{inx} \,:\,
     \Vert \phi \Vert_{\Hg^\gamma}^2 = 
    \vert\phi_0\vert^2+\sum_{n=-\infty}^\infty \vert n\vert^{2\gamma }\vert\phi_n\vert^2<\infty \Bigr\}.
  \]
  Let $\Omega_{N, K }
   =\{ \phi \in {\Hg}^\gamma \,:\, \Vert \phi \Vert^2_{\Lp^2}\leq N \text{\ and\ } \Vert \phi\Vert_{{\Hg}^\gamma}
      \leq K\}$ and let $\Omega_{N,\infty}=\cup_{K=1}^\infty \Omega_{N,K}$.
 
\begin{theorem}\label{thm1.1}
The sets $\Omega_{N,K}$ ($K = 1,2,\dots$) form an
increasing sequence of convex and compact subsets of $\Omega_N$ such that:
\begin{enumerate}
\item[(i)] the Cauchy problem for $\NLSE(p,\beta )$ is well posed for initial datum $\phi\in\Omega_{N,\infty}$ for all $\beta\in {\mR}$ and all $2\leq p<6$;
\item[(ii)] $\Omega_{N,\infty}$ is invariant under the flow associated with $\NLSE (p,\beta )$, and  $\Omega_{N,\infty}$ supports the Gibbs measure on $\Omega_N$;
\item[(iii)] the Gibbs measure on $\Omega_{N,K}$ satisfies a logarithmic Sobolev inequality with some constant $\alpha (N,K)>0$.
\item[(iv)] For $2\leq p\leq 4$, the LSI holds on $\Omega_N$ itself with $\alpha_N>0$.
\end{enumerate}
\end{theorem}

The proof of Theorem~\ref{thm1.1} occupies most of section 2. A version of (iv) with additional hypotheses on $N$ appeared in \cite{GB}. Consequences for the spectral theory of 
Dirac operators with random potentials appear in section 3, where we show that some spectral data can be described by linear statistics on a suitable space of test functions. 
The final section 4 gives consequences for the spectrum of Hill's equation, where we present linear statistics that define Lipschitz functions.\par
\indent Various works concerning concentration of 
Gibbs measure for nonlinear Schr\"odinger equations appear in the literature; for instance, McKean and Vaninsky \cite{MV} consider the defocussing cubic Schr\"odinger equation. 
The authors of 
\cite{RJOT} consider large deviations relating to an invariant measure for a modified cubic Schr\"odinger equation in which 
the nonlinear factor $\beta \vert u\vert^2$ in (\ref{1.3}) is replaced by the bounded factor $\beta \vert u\vert^2/(1+\vert u\vert^2).$

All the measures that we consider will be Radon, that is inner regular and Borel, and we will be considerate about 
the compact sets which almost support the measure. Given a bounded and positive measure, we often 
use $Z$ to stand for the measure of the full space, and we multiply by $Z^{-1}$ to rescale the measure to become a probability.

\section{Convexity of the potentials}

\begin{definition} Let $\Omega$ be a convex subset of ${\Hg}^\gamma$, and $H:\Omega\rightarrow {\mR}$ be continuous. Then $H$ is uniformly convex on $\Omega$ with respect to $\Vert \cdot \Vert^2_{{\Hg}^\gamma}$ if there exists $\eta>0$ such that
  \begin{equation}\label{2.1}
  sH(\phi )+tH(\psi )-H(s\phi +t\psi )
     \geq \eta s t \Vert \phi -\psi \Vert^2_{{\Hg}^\gamma}\qquad (\phi,\psi \in \Omega)
   \end{equation}
for all $0<s,t$ such that $s+t=1$.
\end{definition}

The proof of Theorem~\ref{thm1.1} depends upon the uniform convexity of the Hamiltonian, or perturbations of the Hamiltonian; in turn, this reduces to an
elementary computation of the second derivative of the Hamiltonians. We consider first the nonlinear term.
Let
  $\phi (x)=\sum_{n=-\infty}^\infty (a_n+ib_n)e^{inx}$,
  $\psi (x)=\sum_{n=-\infty}^\infty \xi_n e^{inx}$
   and $\theta (x)=\sum_{n=-\infty}^\infty \eta_n e^{inx}$.
Throughout, $\langle \cdot,\cdot \rangle$ denotes the inner product on a complex Hilbert space, linear in the first variable and conjugate linear in the second variable.

\begin{lemma}\label{lem2.1} Let
  $V(\phi )=\int_{\mT}\vert \phi (x)\vert^p \frac{dx}{2\pi}$.
Then $V$ is a convex function of the real variables $((a_j), (b_j))$ with Hessian matrix satisfying
  \begin{equation}\label{2.2}
  \begin{split}
  \Big\langle {\mathrm{Hess}}(V)_\phi & \svec{\psi}{\theta},
                                       \svec{\psi}{\theta} \Bigr\rangle \\
    & = \frac{p}{2} \int_{\mT}\vert \phi (x)\vert^{p-2}
       \  \Bigl\Vert \svec{\psi (x)+i\theta (x)}{\psi (-x)+i\theta (-x)}
                                            \Bigr\Vert^2 \frac{dx}{2\pi}  \\
   &\qquad + \frac{p(p-2)}{4} \int_{\mT} \vert \phi (x)\vert^{p-4}
            \Bigl\vert\Bigl\langle \svec{\phi (x)}{\overline{\phi} (x)} ,
                 \svec{\overline{\psi} (x)-i\theta (x)}{\overline{\psi}(-x)+i\overline{\theta} (-x)}
                                  \Bigr\rangle\Bigr\vert^2\, \frac{dx}{2\pi}.
   \end{split}
  \end{equation}
\end{lemma}

\begin{proof} This is an elementary computation of the matrix of second order partial derivatives with respect to $a_n$ and $b_n$, which we leave to the
reader (following perhaps the proof of Theorem 3 of \cite{GB}). We are using the real and imaginary Fourier coefficients $a_n$ and $b_n$ of $\phi$, which explains the unexpected appearance of $\psi (-x)$ on the right-hand side.
\end{proof}

\begin{proposition}\label{prop2.2}

  \begin{enumerate}
    \item[(i)] Let $\beta, \gamma >0$ and $2\le p<6$. Then $H$ is uniformly convex with respect to
      $\Vert \,.\,\Vert^2_{{\Hg}^1}$ on $\Omega_N$.
    \item[(ii)] For all $\beta <0$ and $2\leq p<6$, there exists $1/4<\gamma <1/2$ such that for all $N,K>0$ there exists a bounded and continuous $W_K: \Omega_{N,K} \rightarrow {\mR}$ such that $H_K=H+W_K$ is uniformly convex with respect to $\Vert \cdot \Vert_{{\Hg}^\gamma}^2$ for $\eta =1/4$.
  \end{enumerate}
\end{proposition}

\begin{proof} (i) Here the Hamiltonian, $H=(1/2)\int_{\mT}\vert\phi'(x)\vert^2\, dx/2\pi + (\beta /p)V(\phi)$, is the sum of two convex terms, and the inequality (\ref{2.1}) follows from the parallelogram law applied to the term $(1/2)\int_{\mT}\vert\phi'(x)\vert^2 \,dx/2\pi$ in ${\Hg}^1$.
This deals with the defocussing case.\par
\indent (ii) In the focussing case, we need to balance the uniform convexity of $(1/2)\int\vert\phi'\vert^2$ against the concavity of $(\beta /p)V(\phi )$. First we choose
$1/4<\gamma <1/2$ and $C_\gamma$ such that ${\Hg}^\gamma$ is contained in $\Lp^{p-2}$ with
$\Vert \phi\Vert_{\Lp^{p-2}}\leq C_\gamma \Vert \phi\Vert_{{\Hg}^\gamma}$ for all $\phi\in {\Hg}^\gamma$.
Then we choose $1/2<\delta <1$ such that ${\Hg}^\delta$ is contained in $\Lp^\infty$ with
$\Vert \phi\Vert_{\Lp^\infty}\leq C_\delta \Vert \phi\Vert_{{\Hg}^\delta}$ for all $\phi\in {\Hg}^\delta$.
Next we choose $M$ to be the smallest integer that is larger than
$(1+\vert \beta\vert C_\gamma \kappa_pK^{p-2})^{1/(2(1-\delta ))}$, where $\kappa_p>0$ is to be chosen, and then introduce the function
  \begin{equation}\label{2.3}
W_K ((a_n,b_n)_{n=-\infty}^\infty )=(1+\vert\beta\vert C_\gamma \kappa_p K^{p-2} )M^{2\delta}\sum_{j=-M}^M (a_j^2+b_j^2),
  \end{equation}
which is continuous on $\ell^2$ and satisfies the bounds
$0\leq W_K \leq (1+\vert\beta \vert C_\gamma \kappa_pK^{p-2} )M^{2\delta} N$ on
$\Omega_N$.

The functional $F_K(\phi )=(1/2)\int\vert\phi'\vert^2 +W_K(\phi )$ therefore satisfies
  \begin{align}\label{2.4}
   \Bigl\langle {\mathrm{Hess}}(F_K)_\phi \svec{\psi}{\theta}, \svec{\psi}{\theta} \Bigr\rangle
   &=\sum_{j=-M}^M (1+\vert \beta \vert C_\gamma \kappa_pK^{p-2}) M^{2\delta} (\vert \xi_j\vert^2+\vert\eta_j\vert^2)
      \notag\\
   &\qquad \qquad + \sum_{j=-\infty}^\infty j^2(\vert\xi_j\vert^2 +\vert\eta_j\vert^2) \notag\\
   &\geq (1+\vert\beta\vert C_\gamma \kappa_pK^{p-2})\sum_{n=-\infty}^\infty \vert n\vert^{2\delta }(\vert\xi_n\vert^2+\vert\eta_n\vert^2).
  \end{align}

This deals with the concavity of $(\beta/p) V(\phi)$, since by Lemma~\ref{lem2.1} there exist constants $\kappa$ and $\kappa_p$ such that
  \begin{align}\label{2.5}
 {\frac{\vert\beta\vert}{p}}\Bigl\langle {\mathrm{Hess}}(V)_\phi \svec{\psi}{\theta}, \svec{\psi}{\theta}\Bigr\rangle
     &\leq {\frac{\vert\beta\vert\kappa }{p}}\Vert \phi\Vert_{\Lp^{p-2}}^{p-2}(\Vert \psi\Vert^2_{L^\infty}
                      +\Vert\theta\Vert^2_{\Lp^\infty}) \notag\\
     &\leq {\frac{\vert\beta\vert\kappa_p}{p}}\Vert\phi\Vert_{{\Hg}^\gamma}^{p-2}(\Vert\psi\Vert_{{\Hg}^\delta}^{2}
             +\Vert\theta\Vert_{{\Hg}^\delta}^{2})\notag\\
           &\leq {\frac{\vert\beta\vert\kappa_p}{p}}K^{p-2}\sum_{n=-\infty}^\infty \vert n\vert^{2\delta }(\vert\xi_n\vert^2+\vert\eta_n\vert^2). 
  \end{align}
Hence $H_K=F_K+(\beta /p)V$ is uniformly convex with respect to $\Vert \cdot \Vert^2_{{\Hg}^\delta}$ on $\Omega_{N,K}$, and hence also uniformly
convex with respect to $\Vert \cdot \Vert^2_{\ell^2}$.
\end{proof}

{\par\addvspace{6pt plus2pt}
\noindent\prooffont{\bf Proof of Theorem~\ref{thm1.1}\,:}\hskip6pt\ignorespaces}
(i) Bourgain \cite{B} showed that the Cauchy problem is well
posed on ${\Hg}^\gamma$.

(ii) Bourgain also showed that the solution operator for the Cauchy problem is Lipschitz continuous for the $\Lp^2$ norm on each $\Omega_N$. There exists
 $\kappa (N,p)>0$ such that the solution $u(x,t)$ with initial condition $\phi (x)\in \Omega_N$ and the solution $v(x,t)$ with initial condition $\psi (x)\in \Omega_N$ satisfy $\Vert u(\, ,t)-v(\,,t)\Vert_{{\Hg}^\gamma}\leq \kappa (N,p)^t\Vert \phi-\psi\Vert_{{\Hg}^\gamma}$. Hence
$\phi\in \Omega_{N,K}$ implies $u(\,,t)\in \Omega_{N, \kappa (N,p)^tK}$, so each $\Omega_{N,K}$ is progessively mapped into another compact and convex subset $\Omega_{N, \kappa (N,p)^tK}$. Hence $\Omega_{N,\infty}$ is invariant under the flow.\par
\indent (iii) A standard result due to Bakry and Emery \cite{BE} states that any uniformly convex potential $H$ on Euclidean space gives rise to a logarithmic Sobolev inequality for $e^{-H(x)}dx$.  In our case, we can use 
Proposition~\ref{prop2.2} to show that the Hamiltonians $H_K$ are uniformly convex, with constant $\eta\geq 1/4$, and then follow the proof in \cite{BL} which is based 
upon the Pr\'ekopa--Leindler inequality. Given $LSI(\eta )$ for the modified $H_K$, we can use the Holley--Stroock Lemma (\cite{HS} and \cite[Section 9.2]{V}) to recover a logarithmic Sobolev inequality for the original Hamiltonian $H=H_K-W_K$ with constant
  \[ \alpha \geq \eta \exp \bigl({-2 \Vert W_K\Vert_{\Lp^\infty}}\bigr)
      \geq (1/4)\exp \bigl(-2(1+\vert\beta \vert C_\gamma \kappa_pK^{p-2} )M^{2\delta} N\bigr).\]

\indent (iv) The final part follows similar lines to the proof of Proposition \ref{prop2.2}(ii). Observe that, by Lemma~\ref{lem2.1}, 
  \begin{equation}\label{2.7}
  \begin{split}
     \frac{\vert\beta\vert}{2} \Big\langle {\mathrm{Hess}}(V)_\phi \svec{\psi}{\theta},
        &\svec{\psi}{\theta}\Bigr\rangle \\
        &\leq \frac{\vert \beta\vert C}{2}\Vert \phi\Vert^2_{\Lp^2}
            \bigl(\Vert \psi\Vert^2_{\Lp^\infty}+\Vert \theta\Vert^2_{\Lp^\infty}\bigr)\\
        &\leq \frac{1}{2} \sum_{j=-\infty}^\infty \vert\beta\vert C_\delta N\vert j\vert^{2\delta}
               \bigl( \vert\xi_j\vert^2+\vert\eta_j\vert^2\bigr)\\
        &\leq \frac{1}{2} \sum_{j=-\infty}^\infty \bigl( \delta j^2+(1-\delta )
               (\vert \beta\vert C_\delta N)^{1/(1-\delta )}\bigr)
               \bigl( \vert \xi_j\vert^2+\vert\eta_j\vert^2\bigr),
  \end{split}
  \end{equation}
where we have used H\"older's inequality at the last step. To deal with the final term, we 
introduce the functional
   \begin{equation}\label{2.8}
    Y_N(\phi ) =(1-\delta )(\vert \beta\vert C_\delta N)^{1/(1-\delta )}\int_{\mT}\vert\phi (x)\vert^2\, \frac{dx}{2\pi},
   \end{equation}
which is bounded and continuous on $\Omega_N$ and which satisfies the bounds
$0\leq Y_N\leq (1-\delta )(\vert \beta\vert C_\delta )^{1/(1-\delta )}N^{(2-\delta )/(1-\delta )}$.
The perturbed Hamiltonian
   \begin{equation}\label{2.9}
  G_N(\phi )= \frac{1}{2} \int_{\mT}\vert\phi'(x)\vert^2 \, \frac{dx}{2\pi}
                +Y_n(\phi) + \frac{\beta}{2} V(\phi)
   \end{equation}
then satisfies the uniform convexity condition
  \begin{equation}\label{2.10}
  \Big\langle {\mathrm{Hess}}(G_N)_\phi \svec{\psi}{\theta}, \svec{\psi}{\theta} \Bigr\rangle
    \geq \frac{1}{2} (1-\delta)
       \sum_{j=-\infty}^\infty j^2(\vert\xi_j\vert^2+\vert \eta_j\vert^2 ).
   \end{equation}
Hence $H=G_N-Y_N$ is a bounded perturbation of a uniformly convex potential, so using the Holley--Stroock Lemma as in (iii) above, 
we deduce that $\mu^{\beta}_N$ satisfies $LSI(\alpha (N,\beta ))$ on $\Omega_N$, for some $\alpha (N,\beta )>0$.
{ \ifblogo\hskip1.2pt
            \QEDblogo
   \else
   \ifnologo
   \else
   \hfill
            \QEDlogo
   \fi\fi
\par\addvspace{6pt plus2pt}\global\topprheadfalse}

The case $p=4$ is often referred to as the cubic nonlinear Schr\"odinger equation, since $\vert u\vert^2u$ appears in the differential equation. In the next section, we investigate the
associated Dirac operator with random potential $\phi$.

\section{Concentration for the spectral data of Dirac's equation}

Let $\phi =Q+iP$ for $P,Q\in \Lp^2({\mT}, {\mR})$ and consider Dirac's equation $(D-\lambda /2)\Psi_\lambda =0$, or in matrix form
  \begin{equation}\label{3.1}
   \Bigl( \begin{bmatrix} 0 && 1 \\ -1 && 0 \end{bmatrix} \frac{\partial\ }{\partial x}
      - \begin{bmatrix} P && -Q \\ -Q && -P \end{bmatrix}
      -\frac{\lambda}{2} \begin{bmatrix} 1 && 0 \\ 0 && 1 \end{bmatrix}\Bigr)  \Psi_\lambda(x) = 0,
  \end{equation}
for $2\times 2$ matrices with initial condition 
  \[ 
     \Psi_\lambda (0) =\begin{bmatrix} 1 && 0 \\ 0 && 1 \end{bmatrix}.
  \]
Let the characteristic function be $\Delta_\phi (\lambda )={\hbox{trace}}\, \Psi_\lambda (2\pi )$, and introduce the
zeros $\lambda_j$ of $\Delta (\lambda )^2-4=0$ and the zeros $\lambda_j'$ of $\Delta' (\lambda )=0$. The spectral data consist of the $\lambda_j$ and $\lambda'_j$ for $\phi (x)$ and the
corresponding quantities for the translation $\phi (x+t)$. The spectral data partially, but not completely, determines $\phi$. Moser describes the effect on the spectrum of translating the potential in terms of C. Neumann's problem regarding the motion of a particle on a sphere 
subject to a quadratic potential \cite{M}; in our case, the sphere has infinite dimension. 
In the classical context of smooth $\phi$, the spectral data satisfy special estimates regarding the position of the $\lambda_j$. However, for typical $\phi$ in the support of the Gibbs measure,
$\phi$ is not differentiable, and the classical results are inapplicable. Nevertheless, McKean and Vaninsky \cite{MV} developed a spectral theory for $\phi\in \Lp^2$, 
which integrates $\NLSE (4,\beta )$ in terms of infinitely many action and angle variables, and consider the invariant measure in the defocussing case where $\beta >0$.\par 
\indent In this section we introduce statistics which describe the spectral data for
random $\phi$. The statistics define Lipschitz functions on $(\Omega_N, \mu^{\beta}_N)$, and hence by the concentration of measure phenomenon are tightly concentrated about their mean values.\par
\indent For $b>0$, let $S_b$ be the horizontal strip $S_b=\{ z\in {\mC}: \vert \mathop{\mathrm{Im}} z\vert \leq b\}$, and introduce the real Banach space, for the supremum norm,
  \begin{equation}\label{3.3}
  {\RHg}^\infty (S_b)
   = \{ g:S_b\rightarrow {\mC} \,:\, g \text{\  is holomorphic},\ g(\bar z)=\overline {g(z)}
           \text{\ and\ } \sup_{z\in S_b}\vert g(z)\vert<\infty \}.
   \end{equation}

\begin{lemma}\label{lem3.1} For $0\leq k<\infty$ and all $N<\infty$, the mapping $\phi \mapsto (\Delta^{(j)}_\phi (0))_{j=0}^k$ is a
Lipschitz function from $\Omega_N$ to $\ell^2$. The function $\phi \mapsto \Delta_\phi (\lambda)$ is continuous from the norm topology to the topology of uniform convergence on compact sets.
\end{lemma}

\begin{proof} For all $\lambda\in {\mC}$ such that $\vert\lambda\vert \leq M$, we solve the integral equation
  \begin{equation}\label{3.2}
  \Psi_\lambda (x) = \exp\Bigl( \frac{\lambda x}{2}
     \begin{bmatrix} 0 & -1 \\ 1 & 0 \end{bmatrix} \Bigr)
       + \int_0^x \exp\Bigl( \frac{\lambda (x-s)}{2}
       \begin{bmatrix} 0 & -1 \\ 1 & 0 \end{bmatrix}\Bigr)
       \begin{bmatrix} Q(s) & P(s)\\ P(s) & -Q(s) \end{bmatrix} \Psi_\lambda (s)\, ds
  \end{equation}
where the norm of the exponential matrix is uniformly bounded for $0\leq x\leq 2\pi$ and all $\lambda\in S_b$. Indeed, for all $b>0$, the functions $\lambda\mapsto e^{ix\lambda}$ with $x\in [0,2\pi ]$ are uniformly bounded on $S_b$. First we let $0<\delta <1/2(N+M+1)$ and apply the contraction mapping principle
to the right-hand side, as a function of $\Psi_\lambda\in \Lp^2[0,\delta ]$; then we use $\Psi_\lambda (\delta)$ as the
initial value for the corresponding integral equation on $[\delta ,2\delta ]$ after $2\pi/\delta$ steps, we obtain a solution for all $x\in [0,2\pi ]$.\par
\indent By Morera's theorem, $\lambda \mapsto \Psi_\lambda (x)$ is analytic on ${\mC}$, and by the Cauchy--Schwarz inequality we have a bound
  \begin{equation}
   \Vert\Psi_\lambda (x)\Vert^2
   \leq 2e^{2\pi b}
        +4e^{2\pi b}\int_0^x \vert\phi (s)\vert^2\, ds
             \int_0^x \Vert\Psi_\lambda (s)\Vert^2\, ds,
  \end{equation}
where $\int_0^x\vert\phi (s)\vert^2\, ds\leq N$ for all $\phi\in\Omega_N$; similar bounds hold for the $\lambda$ derivatives. Hence by Gronwall's inequality, the maps $\phi \mapsto (d/d\lambda )^j \Psi_\lambda (x)$ are
Lipschitz for $j=0, \dots, k$, hence $\phi \mapsto \Delta^{(j)}(\lambda )$ is Lipschitz. Hence we can introduce $\alpha_j>0$ such that
$\phi \mapsto (\alpha_j \Delta^{(j)}_\phi (0))$ is Lipschitz from $\Omega_N$ to $\ell^2$. By Vitali's convergence theorem, we deduce that
$\phi\mapsto \Delta_\phi (\lambda )$ is continuous for the topology of uniform convergence on compact sets.
\end{proof}
\begin{definition}
Let $D$ be the Dirac operator in $\Lp^2({\mT}; {\mC}^2)$,
so that $D$ is self-adjoint with eigenvalues $\lambda_j$.
Then we define $\Lambda: {\RHg}^\infty (S_b)\rightarrow {\mR}$ to be the bounded linear functional
on the space of test functions given by
  \begin{equation}\label{3.4}
  \Lambda (g)= \sum_{j=-M}^M g(\lambda_j)\qquad (g\in {\RHg}^\infty (S_b)),
  \end{equation}
which is associated with $D$ via $(\lambda_j)_{j=-M}^M.$ Recalling that $D$ depends upon $\phi$, we have a random variable $\phi \mapsto \Lambda_\phi (g)$ on $(\Omega_N, \mu^{\beta}_N)$, called the \emph{linear statistic}.
  \end{definition}
   In statistical mechanics, the term \emph{additive observable} \cite{PS} is used for linear statistic. We pause to observe that, if $(1+z^2)g(z)\in {\RHg}^\infty (S_b)$, then the Fourier transform $\hat g(t)$  is of exponential decay as $t\rightarrow\pm\infty$, so we can write 
  \begin{equation}\label{3.5}  {\text{trace}}\, g(D)
              = {\text{trace}}\int_{-\infty}^\infty e^{itD}\hat g(t)\, \frac{dt}{2\pi}
  \end{equation}
in the style of the Poisson summation formula, and thus obtain the limit of (\ref{3.4}) as $M\rightarrow\infty$. Thus the linear statistic is a generalized trace formula.

\begin{proposition}\label{prop3.1} For all $M < \infty$ there exists $N_M>0$ such that for all  $g\in {\RHg}^\infty   (S_b)$ there exists $\eta_g>0$ such that:
\begin{enumerate}
\item[(i)] the linear statistic $\Lambda_{\phi }(g)= \displaystyle \sum_{j=-M}^M g(\lambda_j')$, where $\Delta_\phi'(\lambda_j')=0$, satisfies
  \begin{equation}\label{3.6}
    \int_{\Omega_{N_M}}\exp\Bigl( t\Lambda_{\phi}(g)-t\int \Lambda (g)\, 
         \frac{d\mu^{\beta}_N}{Z} \Bigr)\, \frac{\mu^{\beta}_N (d\phi)}{Z}
         \leq \exp (\eta_g t^2)\qquad (t\in {\mR});
  \end{equation}
\item[(ii)] a similar result holds for the principal series of eigenvalues given by $\Delta_\phi (\lambda_{2j})=0$,
\item[(iii)] and likewise for the complementary series of eigenvalues given by $\Delta_\phi (\lambda_{2j-1})=0$.
\end{enumerate}
\end{proposition}

\begin{proof} By \ref{thm1.1}(iv), for all $N_M>0$, the measure $\mu^{\beta}$ on $\Omega_{N_M}$ satisfies $LSI(\alpha (N_M))$ for some
$\alpha(N_M)>0$. Hence by section 9.2 of \cite{V}, any
$K$-Lipschitz function $\Phi :\Omega_{N_M}\rightarrow {\mR}$ such that
$\vert \Phi (\phi )-\Phi (\psi )\vert\leq K\Vert\phi-\psi\Vert_{\Lp^2}$ for all $\psi,\psi\in \Omega_{N_M}$
and $\int_{\Omega_{N_M}}\Phi (\phi )\, \mu^{\beta}_N (d\phi )/Z =0$ satisfies the concentration inequality
  \begin{equation*}
  \int_{\Omega_{N_M}}\exp ({t\Phi (\phi )})\, \frac{\mu^{\beta}_{N_M}  (d\phi )}{Z}
      \leq \exp (K^2t^2/\alpha (N_M))\qquad (t\in {\mR}).
  \end{equation*}
So we need to check that the restricted linear statistics give us Lipschitz functionals.

Consider first the case of $\phi =0$. Let $0<r_j<\min \{ b,1/4\}$, and let $C(\lambda'_{j,0}; r_j)$ be the circles with radii $r_j$  that are centred at the zeros of $\Delta_0'$,
let $C=\oplus_{j=-M}^M C(\lambda_{j,0}', r_j)$ be a finite chain of circles. By Lemma~\ref{lem3.1}, there exists $N_M>0$ such that $\vert\Delta'_\phi (\lambda )-\Delta_0'(\lambda )\vert <(1/2)\vert\Delta'_0(\lambda )\vert$
for all $\lambda$ on $C$; Now $\Delta_0'(\lambda )=0$ has only simple zeros, and likewise for $\Delta_\phi'$ by Rouch{\'e}'s theorem; hence by the calculus of residues the restricted linear statistic associated with $\phi$ satisfies
  \begin{equation}\label{3.7}
   \sum_{j=-M}^M g(\lambda'_j)
     = \frac{1}{2\pi i} \int_C g(\lambda ) \frac{\Delta_\phi ''(\lambda )}{\Delta_\phi '(\lambda )} \, d\lambda .
  \end{equation}
Hence by Lemma~\ref{lem3.1}, the map $\phi\mapsto \sum_{j=-M}^M g(\lambda'_j)$ is Lipschitz in a neighbourhood of $\phi =0$.

 The proofs of (iii) and (iii) are similar.
  \end{proof}
To clarify the connection between the spectral data of Dirac's equation and the $\NLSE$ we consider  
  \begin{align}\label{3.9}
   0=\Bigl(&\begin{bmatrix} 0 & 1 \\ -1 & 0 \end{bmatrix}
      \frac{\partial\ }{\partial t}
      - \begin{bmatrix} Q'-(\beta /2)(P^2+Q^2) & P' \\ P' & -Q'-(\beta /2)(P^2+Q^2) \end{bmatrix}\cr
      &- \lambda \begin{bmatrix} -P & Q \\ Q & P \end{bmatrix}
      + \frac{\lambda^2}{2}    \Bigr) \Psi_\lambda (x,t)
  \end{align}
where $P'$ denotes ${\frac{\partial P}{\partial x}}$. Its \cite{MV} observed that, when $\phi$ satisfies $\NLSE(4,\beta )$, 
the pair of equations (\ref{3.1}) and (\ref{3.9}) are compatible.

\section{Concentration of the spectral data for Hill's equation}

Suppose that $q\in \Lp^2({\mT};{\mR})$ is $\pi$-periodic, and note that Dirac's equation (\ref{3.1}) reduces with 
$P=( 1+q)/2$ and $Q=i(1-q)/2$, to Hill's equation in the form
\begin{equation}\label{4.1} -f''+qf=\lambda f.\end{equation}
 The periodic spectrum of Hill's equation 
consists of those $\lambda$ such that a non-trivial periodic solution exists for (\ref{4.1}), and it is known that the 
 corresponding eigenvalues are $\lambda_0<\lambda_1\leq \lambda_2<\lambda_3\leq \lambda_4<\dots$, namely the zeros of $\Delta (\lambda)^2-4=0$. This data partially, 
 but not completely, determines $q$; see \cite{MT}. \par

The intervals $(\lambda_{2j-1}, \lambda_{2j})$ are known as the spectral gaps, or intervals of instability, and Erdelyi showed as in (9.55) of \cite{B2}  that the sequence of gap lengths 
$(\lambda_{2j}- \lambda_{2j-1})_{j=1}^\infty\in \ell^2$; this estimate cannot be much improved for typical $q\in \Lp^2$. We now consider the midpoints $(\lambda_{2n-1}+\lambda_{2n})/2$ of these gaps, and form the sequence

\begin{align} t_n&=\sqrt{(\lambda_{2n-1}+\lambda_{2n})/2}\qquad(n=1,2, \dots );\cr
 t_0&=0;\cr
 t_{n}&=-\sqrt{(\lambda_{-(2n+1)}+\lambda_{-2n})/2}\qquad
 (n=-1,-2, \dots ).\end{align}
\indent The space of potentials that have a 
 given periodic spectrum can be parameterized by the real torus ${\mT}^\infty$, which arises from a spectral curve \cite{MT}.
 The KdV equation gives rise to an evolution which preserves the periodic spectrum of Hill's equation. The Hamiltonian
\begin{equation}H(q)={\frac{1} {2}}\int_{\mT}(q'(x))^2{\frac{dx}{2\pi}} -{\frac{\beta}{6}}\int_{\mT} q(x)^3\, 
{\frac{dx}{2\pi}}\end{equation}
has canonical equations of motion which give the periodic KdV equation 
${\frac{\partial u}{\partial t}}+{\frac{\partial^3u}{\partial x^3}}+\beta u{\frac {\partial u}{\partial x}}=0$ 
and as above, we take the phase space to be $\Omega_N=\{ q\in L^2({\mT};{\mR}):\int_{\mT} q(x)^2/2\pi\leq N\}$. 
Bourgain \cite{B2} introduced a Gibbs measure $\nu_N^\beta $ on $\Omega_N$, and showed that the Cauchy problem is well posed on the support of $\nu_N^\beta$. 
In this section, we consider the periodic spectral data for $u$ in the support of the Gibbs measure. The definition of $\nu_N^\beta$ is analogous to the construction 
of $\mu^{\beta}_N$ in the introduction, and is presented 
in more detail in \cite{B} and \cite{GB}.\par

For $b>0$, we recall the Paley--Wiener space $\PW(b)$ of entire functions $g$ of exponential type such that $g\in \Lp^2({\mR}; {\mC})$ and 
\begin{equation}\label{3.8}
  \limsup_{y\rightarrow\mp\infty} \frac{\log\vert g(iy)\vert}{\vert y\vert}\leq b,
\end{equation}
so that $\PW(b)$  is a complex Hilbert space for the usual inner product on $\Lp^2({\mR};{\mC})$, with a real linear subspace 
$\RPW(b)=\{ g\in \PW(b): \bar g(z)=g(\bar z)\, \forall z\in {\mC}\}$.

\begin{proposition}\label{prop4.1}
\begin{enumerate}\item[(i)]  Suppose that $\int_{\mT} q (x)dx/2\pi =0$ and $\int_{\mT}\vert q(x)\vert dx/2\pi <1/2$. Then 
there exist constants $A,B>0$ such that
  \begin{equation}\label{eqn29} 
    A\Vert g\Vert_{\Lp^2}^2
    \leq \sum_{n=-\infty}^\infty \vert g(t_n)\vert^2\leq B\Vert g\Vert^2_{\Lp^2}
        \qquad (g\in \PW(2)),
   \end{equation}
   and $\sum_{j=-\infty}^\infty (g(t_j)-g(j))$ is absolutely convergent for all $g\in \RPW(2)$.
\item[(ii)] There exists $N>0$ such that, for all $g\in \RPW(2)$, the linear statistics $\Lambda_q (g)=\sum_{n=-M}^M g(t_n)$ satisfy the concentration of measure 
condition as in (\ref{3.6}) with respect to Gibbs measure $\nu_N^\beta $ on $\Omega_N$.
\end{enumerate}
\end{proposition}
\begin{proof} (i) To recover a function in $\PW (2)$ from its values at a sequence of equally spaced real points, the optimal sampling rate is the Nyquist rate $\pi /2$. By sampling more frequently, we can accommodate 
the irregularity of the points $(t_n)$. From Borg's estimates, one can show that the sequence satisfies $\vert t_n-n\vert <1/4$ for all integers $n$; 
hence, the sequence $(t_n)$ is uniformly discrete and satisfies the sampling conditions 
$t_n\rightarrow\pm\infty$ as $n\rightarrow\pm \infty$, $t_{n+1}-t_n <3/2$ and $t_n-t_m>1/2$ for all $n>m$. Then one can invoke the sampling theorem in Corollary 7.3.7 of 
\cite{P} which gives the stated result (\ref{eqn29}).\par
\indent Each $g\in\RPW (2)$ is band limited, in the sense that the Fourier transform is supported on $[-2,2]$, and likewise $g'$ is band limited. By the mean value theorem, for each $j$, there exists $t'_j$ between $j$ and $t_j$ such that $g(t_j)-g(j)=g'(t_j')(t_j-j)$, and the sequence $(t_j')$ is likewise uniformly discrete and sampling. Since 
$g'\in \RPW(2)$, we can write
\begin{align}\sum_{j=-\infty ;j\neq 0}^\infty \vert g(t_j)-g(j)\vert &\leq\Bigl(\sum_{j=-\infty ;j\neq 0}^\infty \vert g'(t_j')\vert^2\Bigr)^{1/2}\Bigl(\sum_{j=-\infty ;j\neq 0}^\infty
{\frac{(t_j^2-j^2)^2}{(t_j+j)^2}}\Bigr)^{1/2}\cr
&\leq 2\sqrt{B}\Vert g\Vert_{\Lp^2}\Bigl(\sum_{j=-\infty ;j\neq 0}^\infty {\frac{C^2}{j^2}}\Bigr)^{1/2}
\end{align}
since by (9.55) of \cite{B2}, there exists $C$ such that $\vert t_j^2-j^2\vert\leq C$ for all $j$.\par
(ii) Let $\Delta$ be the characteristic function of Hill's equation, which is defined as for Dirac's equation, and hence is entire. The periodic spectrum is given by the 
zeros of $\Delta (\lambda )^2-4=0$,  so for each $g\in\RPW (2)$ we can use Cauchy's integral formula  
    \begin{equation} 
      t_{2n}^2=
        \int_{C(4n^2;1/4)}\frac{\lambda \Delta'(\lambda )}{\Delta (\lambda )-2}
              \frac{d\lambda}{4\pi i}
    \end{equation}
to express the functionals from the principal series, and a corresponding formula for the $t_{2n-1}^2$  from the complementary series. 
Each functional $\phi\mapsto \Delta_\phi  (\lambda )$ is Lipschitz, uniformly on the circle of integration, as in Lemma \ref{lem3.1}, and 
hence $\phi\mapsto g(t_j)$ is also Lipschitz for each $j$ and $g\in \RPW(2)$. In Corollary 2 of \cite{GB}, we proved that $\nu_N$ satisfies a logarithmic Sobolev 
inequality, and hence a concentration of measure inequality for real Lipschitz functions on $\Omega_N$ for $N>0$ suitably small.
\end{proof}

The spectral curve of Hill's equation is 
the transcendental curve 
  \begin{equation} 
    {\mathcal{E}}_\phi=\{ (w, z)\in {\mC}^2: w^2=4-\Delta_\phi (z)^2\}
  \end{equation}
with branch points at the $\lambda_j$. In our case the branch points are random, although they remain close to their mean values with high probability, as shown by the linear statistics in \ref{prop4.1}. To interpret this geometrically, we consider the linear isomorphism between the spaces
\begin{equation} \Bigl\{h\in \Lp^2([-b,b];{\mR}): h(t)=h(-t)\quad\forall t\in [-b,b]; \int_{-b}^b h(t)\, dt=0\Bigr\}\end{equation}
and 
\begin{equation}\bigl\{ g\in \RPW (b): g(z)=g(-z)\quad \forall z\in {\mC}; \quad x^2g(x)\in \Lp^2({\mR})\bigr\}\end{equation}
given by the Fourier transform $g(z)=\int_{-b}^b {\frac{1-\cos zt}{z^2}}h(t)dt$. Take such a $g$ with $b=2;$ then $f(z)=g(\sqrt{z}/2)$ is an entire function of order 
$1/2$ and type $1$, hence belongs to the space $I^{3/2}$ as in \cite{MT}. McKean and Trubowitz consider complex sequences $(x_j)$ such that $\sum_{j=1}^\infty x_j f(\lambda_{2j})$ converges, and interpret such an expression in terms of the Jacobi map on ${\mathcal{E}}_\phi$ into the infinite real torus ${\mT}^\infty$.

\end{document}